\documentclass[letterpaper,11pt]{article}

\usepackage[letterpaper]{geometry}
\usepackage{amsmath, amsthm, amssymb}
\usepackage{graphicx}
\usepackage{fullpage}
\usepackage[pdftitle={Tile},
            pdfauthor={Author},
            pdfkeywords={Keywords}]{hyperref}
\usepackage{mathtools}
\usepackage{bbm}
\usepackage{xspace}

\newtheorem{lemma}{Lemma}
\newtheorem{thm}{Theorem}
\newtheorem{cor}{Corollary}

\newtheorem{fact}{Fact}
\newtheorem{observation}{Observation}
\newtheorem{claim}{Claim}

\newtheorem{myclaim}{Claim}

\newenvironment{proofof}[1]{\noindent{\emph{Proof of {#1}.}}}{}

\newcommand{\initOneLiners}{%
    \setlength{\itemsep}{0pt}
    \setlength{\parsep }{0pt}
    \setlength{\topsep }{0pt}
}

\newcommand{\corner}{\mathcal{C}}
\newcommand{\contCorner}{\mathcal{CC}}

\newcommand{\lpCorner}{\mathcal{C}_{LP}}

\newcommand{\sClosure}{\mathcal{SC}}

\newcommand{\mink}[2]{\gamma_{#1}(#2)}
\newcommand{\minkEmpty}{\gamma}
\newcommand{\R}{\mathbb{R}}

\newcommand{\Z}{\mathbb{Z}}

\newcommand{\pFiniteSupp}[1]{\R^{(#1)}_+}

\DeclareMathOperator{\intr}{int}

\newcommand{\remove}[1]{}

\DeclareMathOperator{\GMI}{GMI}

\newcommand{\conv}{\textrm{conv}}
\DeclareMathOperator{\inte}{int}

\newcommand{\mmnote}[1]{\noindent {\bf Marco: #1}}

\newcommand{\afterThesis}[1]{}

\begin{document}

\title{Characterization of the Split Closure via Geometric Lifting}
\author{Amitabh Basu\footnote{Department of Applied Mathematics and Statistics, Johns Hopkins University.} \and Marco Molinaro\footnote{Department of Industrial and Systems Engineering, Georgia Institute of Technology.}}
\date{}

%
%
%
\maketitle

\begin{abstract}
We analyze {\em split cuts} from the perspective of {\em cut generating functions} via \emph{geometric lifting}. We show that $\alpha$-cuts, a natural higher-dimensional generalization of the $k$-cuts of Cornu\'{e}jols et al., gives all the split cuts for the mixed-integer corner relaxation. As an immediate consequence we obtain that the $k$-cuts are equivalent to split cuts for the 1-row mixed-integer relaxation. Further, we show that split cuts for finite-dimensional corner relaxations are restrictions of split cuts for the infinite-dimensional relaxation. In a final application of this equivalence, we exhibit a family of pure-integer programs whose split closures have arbitrarily bad integrality gap. This complements the mixed-integer example provided by Basu et al [{\em On the relative strength of split, triangle and quadrilateral cuts}, Math. Program. 126(2):281--314, 2011].
%
\end{abstract}




\section{Introduction}

The existing literature on cutting planes for general integer programs can be roughly partitioned into two classes. One class of cutting planes relies on the paradigm of {\em cut-generating functions} (CGFs) \cite{gerardCGF}, whose roots can be traced back to Gomory and Johnson's work in the 70s. \remove{Roughly speaking, in a cut-generating function each cut coefficient is obtained using only local information about the given variable (i.e. the corresponding column in the LP). }Prominent examples in this class are {\em Gomory mixed-integer cuts (GMI), mixed-integer rounding cuts (MIR), reduce-and-split cuts} and \emph{$k$-cuts} \cite{gomoryGMI,NW90,AndersenCL05,CLV}. 
The other class of cutting planes is based on the idea taking convex hulls of disjunctions; this perspective goes back to Balas' work on {\em disjunctive programming}~\cite{balasDisjunctive}. Examples in this class are {\em split cuts, lift-and-project cuts} and, more recently, {\em t-branch disjunctions}~\cite{bcc,tBranch,sherali-adams}.

One important advantage of the cut-generating function paradigm is that often they give closed form formulas to generate cuts that can be computed very efficiently. In contrast, the disjunctive programming paradigm usually involves solving a {\em cut generating linear program} \cite{bcc}, which can be computationally expensive. This leads to a significant blowup in time required to generate cuts using the disjunctive paradigm, as opposed to the cut generating function framework. 
Hence, there is great value in results which show that a certain family of disjunctive cuts can also be obtained by cut generating functions that are computable efficiently.
Moreover, establishing connections between the cut-generating and the disjunctive paradigms contributes to the overarching goal of obtaining a more complete understanding of cutting planes.

In this paper, we analyze the family of {\em split cuts} \cite{CKS} from this viewpoint. Split cuts are one of the most important cuts in practice and their effective use was responsible for a major improvement in MIP solvers since the 90s \cite{AchterbergWunderling13,Bixby}. Due to their importance, these cuts have been extensively studied, both theoretically~\cite{ACL,bcm,NW90} and computationally~\cite{BCCN,balasSaxena,dashGoycoolea,bonami,fischettiSal,SplitApprox}. 

An important tool for the cut generating function approach has been the study of finite and infinite dimensional {\em corner relaxations}. Corner relaxations retain enough of the complexity of mixed-integer programs to be extremely useful models for obtaining general cutting planes, and yet have a structure that yields to mathematical analysis and a beautiful theory has been built about them~\cite{corner_survey}.

In \autoref{sec:prelim}, we formally define the main objects of study of this paper, and then follow up with a statement of our results in \autoref{sec:results}.

\subsection{Preliminaries}\label{sec:prelim}

\subsubsection{Corner Relaxations}\label{s:basic_defs}
		
			Given a point $f \in [0,1]^n \setminus \Z^n$ and sets $R, Q \subseteq \R^n$ (not necessarily finite), define the \emph{($n$-row) generalized corner relaxation} $\corner = \corner(f, R, Q)$ as the set solutions $(x, s, y)$ to the system
			\begin{align}
				x &= f + \sum_{r \in R} r \cdot s(r) + \sum_{q \in Q} q \cdot y(q) \notag\\
				x &\in \Z^n \tag{$\corner(f, R,Q)$} \label{eq:corner}\\
				y(q) &\in \Z \ \ \forall q \in Q \notag\\
				s &\in \pFiniteSupp{R}, y \in \pFiniteSupp{Q}, \notag
			\end{align}
			where $\pFiniteSupp{S}$ denotes the set of non-negative functions $f: S \rightarrow \R_+$ with finite support. This model was first introduced by Johnson in~\cite{johnson}. When $R$ and $Q$ are finite subsets of $\R^n$, the above problem is refered to as the {\em mixed-integer corner polyhedron}. \remove{which was introduced by Gomory in~\cite{corner}. The corner polyhedron has proved to be an extremely important tool for studying general mixed-integer linear programs. {\bf Amitabh : Add references to show evidence of this claim} In this paper, we will simply refer to the problem~\eqref{eq:corner} as the corner relaxation. The $x$ variables are called {\em basic variables} and the $s, y$ vairables are called {\em non-basic} variables.}
	
		We define the {\em continuous corner relaxation} $\contCorner(f, R, Q)$ obtained from~\eqref{eq:corner} by dropping the integrality constraint on the $y$ variables, i.e., all points $(x,s,y)$ that satisfy the first, second and fourth constraints. 
We use $\lpCorner(f, R, Q)$ to denote the linear relaxation of $\corner(f, R, Q)$, namely the set of points $(x, s, y)$ that satisfy the first and fourth set of constraints in \eqref{eq:corner}. Thus we have that $\corner(f,R,Q) \subseteq \contCorner(f,R,Q) \subseteq \lpCorner(f,R,Q)$.
			
	
	
			\subsubsection{Valid cuts.} A \emph{valid cut} or \emph{valid function} for $\corner(f, R, Q)$ (or $\contCorner(f,R,Q)$) is a pair of non-negative functions $(\psi, \pi) \in \R^{R}_+ \times \R^{Q}_+$ (where $\R^{S}_+$ denotes the set of functions from $S$ to $\R_+$)\afterThesis{\mmnote{Define $\R^S$?}} such that for every $(x, s, y) \in \corner(f, R, Q)$ (or $(x, s, y) \in \contCorner(f, R, Q)$ resp.) we have
			\begin{align}
				\sum_{r \in R} \psi(r) s(r) + \sum_{q \in Q} \pi(q) y(q) \ge 1. \label{eq:validCut}
			\end{align}
			We will associate the function pair $(\psi, \pi)$ with the cut defined by the above inequality. With slight abuse of notation, we may use $(\psi, \pi)$ to denote the set of points satisfying \eqref{eq:validCut} (e.g., $\lpCorner(f, R, Q) \cap (\psi, \pi)$). 
		
			\afterThesis{\mmnote{Mention anything about assuming non-negativity in the def of valid functions?}}
			
			Also with slight overload in notation, given sets $R' \supseteq R$ and $Q' \supseteq Q$ and functions $\psi \in \R^{R'}$ and $\pi \in \R^{Q'}$, we say that $(\psi, \pi)$ is a valid cut/function for $\corner(f, R, Q)$ if the restriction $(\psi|_{R}, \pi|_{Q})$ is valid for it. Notice that if $(\psi, \pi)$ is valid for $\corner(f, R', Q')$, then it is valid for the restriction $\corner(f, R, Q)$. 
			

\subsubsection{GMI and $\alpha$-cuts}\label{sec:alpha-def}

We define a family of cut-generating functions for the $n$-row corner relaxation, that we call \emph{$\alpha$-cuts}, which is a natural higher-dimensional generalization of $k$-cuts~\cite{CLV}. Informally, an $\alpha$-cut for an $n$-row corner relaxation is obtained by taking an integer vector $\alpha \in \Z^n$, aggregating the $n$ rows of the problem using the $\alpha$'s as multipliers and then employing the GMI function to the resulting equality.
		
		Given a real number $a \in \R$, let $[a]$ denote its fractional part $a - \lfloor a \rfloor$. Then given $f \in \R$, the \emph{GMI function} \afterThesis{\cite{xx,xx}}$(\psi^f_{\GMI}, \pi^f_{\GMI})$ is defined as
	\begin{equation}\label{eq:defGMIPi}
		\psi^f_{\GMI}(r) = \max\left\{\frac{r}{1-[f]}, -\frac{r}{[f]} \right\}, \qquad\pi^f_{\GMI}(q) = \max\left\{ \frac{[q]}{1-[f]}, \frac{1-[q]}{[f]} \right\}. 
	\end{equation}
	The GMI function is valid for the infinite corner relaxation $\corner(f, \R, \R)$ ($n=1$).
	
	\afterThesis{\mmnote{Add figure of a GMI and a $k$-cut}}
	
	For $f \in [0,1]^n \setminus \Z^n$, define the sets $\Z_f = \{ w \in \Z^n : fw \notin \Z\}$. Given $\alpha \in \Z_f$, we define the \emph{$\alpha$-cut function} $(\psi^f_\alpha, \pi^f_\alpha)$ is defined as	
	\begin{gather}
		\psi^f_{\alpha}(r) = \psi^{\alpha f}_{\GMI}(\alpha r), ~~~~~~~~~~
		\pi^f_{\alpha}(q) = \pi^{\alpha f}_{\GMI}(\alpha q). \label{eq:defAlphaCutPi}
	\end{gather}
	We make the observation here that $\psi^{-f}_{\GMI}(-r) = \psi^f_{\GMI}(r)$ and similarly, $\pi^{-f}_{\GMI}(-q)=\pi^f_{\GMI}(q)$ (since $[-a] = 1 - [a]$). Therefore, $\psi^{f}_{-\alpha} = \psi^f_{\alpha}$ and $\pi^{f}_{-\alpha}(q) = \pi^f_{\alpha}(q)$, and one needs only consider $\alpha \in \Z_f$ lying on one side of any halfspace in $\R^n$. In the case $n = 1$, the $\alpha$-cut functions correspond to the \emph{$k$-cuts} introduced in \cite{CLV}. 
			

	\subsubsection{Split Cuts}
			
		Consider an $n$-dimensional corner relaxation $\corner(f, R, Q)$. Given $\alpha \in \Z_f$ and $\beta : Q \rightarrow \Z$, we define the split disjunction
			\begin{align*}
				D(\alpha, \beta, f) \triangleq \left\{ (x, s, y) : \alpha x + \sum_{q \in Q} \beta(q) y(q) \le \lfloor \alpha f \rfloor \right\} \cup \left\{ (x, s, y) : \alpha x + \sum_{q \in Q} \beta(q) y(q) \ge \lceil \alpha f \rceil \right\}.
			\end{align*}
			\afterThesis{Notice that this is the unique split disjunction with ``normal'' $(\alpha, \beta)$ that contains the point $(f, 0, 0)$.} A cut $(\psi, \pi)$ is a \emph{split cut} for $\corner(f, R, Q)$ with respect to the disjunction $D(\alpha, \beta, f)$ if it is satisfied by all points in the set $\lpCorner(f, R, Q) \cap D(\alpha, \beta, f)$.\footnote{Notice that we only consider $\alpha$ in $\Z_f$, and not in $\Z^n$, because in the latter case $\lpCorner(f, R, Q) \cap D(\alpha, \beta, f) = \lpCorner(f, R, Q)$.}

We analogously define split cuts for $\contCorner(f, R, Q)$. For that, given $\alpha \in \Z_f$, define the split disjunction $$\bar{D}(\alpha, f) = \{(x, s, y) : \alpha x \leq \lfloor \alpha f \rfloor \}\cup \{ (x, s, y) :\alpha x \ge \lceil \alpha f \rceil \}.$$ The cut $(\psi, \pi)$ is a \emph{split cut} for $\contCorner(f, R, Q)$ with respect to the disjunction $\bar{D}(\alpha, f)$ if it is satisfied by all points in $\lpCorner(f, R, Q) \cap \bar{D}(\alpha, f)$. 

Notice that these definitions hold for the case where $R$ and/or $Q$ are infinite, thus generalizing the standard definition of split cuts to this setting. 
	
			\afterThesis{\mmnote{Say that split cuts are very important, etc.}}


\subsection{Statement of Results.}\label{sec:results}
%
	We show that the $\alpha$-cuts give all the split cuts for a generalized mixed-integer corner relaxation, which includes both the finite and infinite dimensional relaxations. 

	\begin{thm} \label{thm:equivalence}
		Consider an $n$-row generalized corner relaxation $\corner$. Then the $\alpha$-cuts are exactly the split cuts for $\corner$. More precisely,
		every $\alpha$-cut is a split cut for $\corner$ and every split cut for $\corner$ is dominated by an $\alpha$-cut.
	\end{thm}
	In particular, this result directly gives the equivalence between $k$-cuts and split cuts for the $1$-row (finite or infinite dimensional) corner relaxations.
	
	The proof of Theorem \autoref{thm:equivalence} is based on \emph{geometric lifting}, a technique introduced in~\cite{ccz}. The cut-generating functions that we obtain can also be derived  using the formulas from Section 2.4 of \cite{AndersenCL05}. However, in our opinion, our geometric viewpoint illuminates aspects of these cut generating functions that are not immediately apparent if one uses the algebraic approach of \cite{AndersenCL05}. We illustrate this advantage using two applications: Theorems~\autoref{thm:universalIR} and~\autoref{thm:badIP} below.
	
 It is known that every cut for the finite-dimensional corner relaxation is a restriction of a cut for the infinite relaxation. We show that, not only does the infinite relaxation encode all finite-dimensional corner relaxations and valid cuts for them, but it also preserves the structure of split cuts. 
	
	\begin{thm}\label{thm:universalIR}
		Consider a non-empty $n$-row generalized corner relaxation $\corner(f, R, Q)$. Then a valid cut for $\corner(f, R, Q)$ is a split cut if and only if it is the restriction of a split cut for the infinite relaxation $\corner(f, \R^n, \R^n)$.
	\end{thm}
	This result relies on the fact that the characterization of Theorem \autoref{thm:equivalence} holds for the infinite relaxation.  
	
	In addition, we use insights obtained from this characterization to better understand the strength of split cuts. Cook, Kannan and Schrijver constructed an example of a \emph{mixed-integer} program with infinite split rank \cite{CKS}. Reinforcing the potential weakness of the split closure, Basu et al. \cite{basuStrength} constructed \emph{mixed-integer} programs whose split closures provide an arbitrarily weak approximation to their integer hull. On the other hand, it is known that every \emph{pure-integer} program (where all variables are integer valued) has finite split rank \cite{finiteCG}. Nonetheless, in our final contribution we construct an explicit family of pure-integer programs whose split closure is arbitrarily bad compared to the integer hull.

	\begin{thm} \label{thm:badIP}
		For every rational $\epsilon > 0$, there is a pure-integer corner relaxation $\corner_\epsilon \subseteq \R^3$ whose split closure $\sClosure_\epsilon$ has the property that
		\begin{align*}
			\frac{\min\{ y_1 + y_2 : (x,y_1, y_2) \in \corner_\epsilon\}}{\min\{y_1 + y_2 : (x,y_1, y_2) \in \sClosure_\epsilon\}} \ge \frac{1}{12\epsilon}.
		\end{align*}
	\end{thm}

\paragraph{Comparison with MIR functions} 
Mixed-integer rounding (equivalently, the Gomory mixed-integer) functions can be seen as a family of cut generating functions that completely describe all split cuts for finite-dimensional mixed-integer corner polyhedra \cite{NW90}. MIR functions are parametrized by tuples of rational numbers; in contrast, our $\alpha$-cuts are parameterized by tuples of integers. This property ascribes an additional simplicity on the cut-generating functions that we obtain. 

	

\section{Lattice-free cuts and split cuts}
		
		\paragraph{Lattice-free split cuts} Now we define the crucial concept of \emph{lattice-free cuts}. First, given a convex body $S \subseteq \R^n$ containing the origin in its interior, define the function \begin{equation}\label{eq:mink}\mink{S}{r} = \inf \{\frac{1}{\lambda} > 0 : \lambda r \in S\}.\end{equation} A convex body $S\subseteq \R^n$ is said to be {\em lattice-free} if $\Z^n \cap \inte(S) = \emptyset$. Given a lattice-free convex  body $S \subseteq \R^n$ containing $f$ in its interior, we define the function $(\psi_S, \pi_S)$ by setting $\psi_S(r) = \pi_S(r) = \mink{S - f}{r}$ for all $r\in \R^n$. The cut $(\psi_S, \pi_S)$ is valid for $\contCorner(f, \R^n, \R^n)$. We call these \emph{lattice-free} cuts.

We define a class of lattice-free cuts using a special kind of lattice-free convex set. Given an integer vector $\alpha \in \Z^n$, we define the \emph{lattice-free split} $S(\alpha, f)$ as
		\begin{align*}
			S(\alpha, f) = \{ x \in \R^n : \lfloor \alpha f \rfloor \le  \alpha x \le \lceil \alpha f \rceil\}.
		\end{align*}
		The set $S(\alpha, f)$ is convex and lattice-free, and notice that whenever $\alpha \in \Z_f$ it contains $f$ in its interior; in this case, $(\psi_{S(\alpha, f)}, \pi_{S(\alpha, f)})$ is a valid cut for $\contCorner(f, R, Q)$. We call this cut a \emph{lattice-free split cut}.
		
		To simplify the notation define $\bar{S}(\alpha, f) = S(\alpha,f) - f$ as the ``centered'' version of $S(\alpha,f)$, so that $\psi_{S(\alpha,f)} = \pi_{S(\alpha,f)} = \minkEmpty_{\bar{S}(\alpha,f)}$. It is well-known that $\minkEmpty_{\bar{S}(\alpha,f)}$ (and thus $\psi_{S(\alpha,f)}$ and $\pi_{S(\alpha,f)}$) are sublinear functions, i.e., $\minkEmpty_{\bar{S}(\alpha,f)}(0) = 0$, $\minkEmpty_{\bar{S}(\alpha,f)}(r^1 + r^2) \leq \minkEmpty_{\bar{S}(\alpha,f)}(r^1) + \minkEmpty_{\bar{S}(\alpha,f)}(r^2)$ for all $r^1, r^2 \in \R^n$, and $\minkEmpty_{\bar{S}(\alpha,f)}(t r) = t \minkEmpty_{\bar{S}(\alpha,f)}(r)$ for all $t \in \R$ and $r \in \R^n$. 

	\paragraph{Equivalence of split and lattice-free split cuts} We will use the fact that, for the continuous corner relaxation, split and lattice-free split cut are equivalent. 
The finite-dimensional version of this result is due to Balas~\cite{intersectionCuts} (see also Section 4.9 in~\cite{cornuejols-survey}).
		
\begin{thm} \label{thm:equivSplitLF}
	For any continuous corner relaxation $\contCorner(f, R, Q)$ and $\alpha \in \Z^n$,	
$$\lpCorner(f, R, Q) \cap (\psi_{S(\alpha, f)}, \pi_{S(\alpha,f)}) = \conv\left( \lpCorner(f, R, Q) \cap \bar{D}(\alpha, f)\right).$$
		\end{thm}
		
\begin{proof}
For ease of notation, we denote $ \lpCorner = \lpCorner(f, R, Q)$, $S = S(\alpha, f)$  and $\bar S = \bar{S}(\alpha, f)$ in the remainder of this proof. 
	
		$(\supseteq)$ Since $\lpCorner \cap (\psi_S, \pi_S)$ is convex, it suffices to show that $\lpCorner \cap \bar{D}(\alpha, f)\subseteq \lpCorner \cap (\psi_S, \pi_S)$; so it suffices to show that $\lpCorner \cap \bar{D}(\alpha, f)$ satisfies $(\psi_S, \pi_S)$.
		
		Take $(\bar{x}, \bar{s}, \bar{y}) \in \lpCorner  \cap \bar{D}(\alpha, f)$. Suppose by contradiction that 
		\begin{align}\label{eq:strictineq}
			\sum_{r \in R} \bar{s}_r \psi_S(r) + \sum_{q \in Q} \bar{y}_q \pi_S(q) < 1.
		\end{align}
		
		Using the sublinearity of $\minkEmpty_{\bar{S}(\alpha,f)} = \psi_S =\pi_S$ and the fact that $\bar{s}$ and $\bar{y}$ have finite support, we obtain from \eqref{eq:strictineq} that $$\minkEmpty_{\bar{S}(\alpha,f)}(\sum_{r \in R} \bar{s}_r r +  \sum_{q \in Q} \bar{y}_q q) < 1.$$ Since $\minkEmpty_{\bar{S}(\alpha,f)}$ is the Minkowski functional of $\bar{S}$, this shows that $\bar x - f = \sum_{r \in R} \bar{s}_r r +  \sum_{q \in Q} \bar{y}_q q \in \intr(\bar{S})$. In other words, $\bar x \in \intr(S)$. This is a contradiction to the choice $(\bar{x}, \bar{s}, \bar{y}) \in \lpCorner  \cap \bar{D}(\alpha, f)$.

		 	\medskip
		 	
		 	$(\subseteq)$ Take a point $(\bar{x}, \bar{s}, \bar{y}) \in \lpCorner \cap (\psi_S, \pi_S)$; we show that it belongs to $\conv(\lpCorner \cap \bar{D}(\alpha, f))$. 
			\begin{claim} \label{claim:belongsS}
		Let $S(\alpha, f)$ be a lattice-free split for some $\alpha \in \Z_f$. Let $r^*$ be such that $\psi_{S(\alpha,f)}(r^*) = \pi_{S(\alpha, f)}(r^*) = 0$. Then a point $\tilde{x} \in \R^n$ belongs to $\inte(S(\alpha, f))$ if and only if $\tilde{x} + r^*$ belongs to $\inte(S(\alpha, f))$.
	\end{claim}
	
	\begin{proof}
		By definition $S(\alpha, f) = \{x : \lfloor \alpha f \rfloor \le  \alpha x \le \lceil \alpha f \rceil\}$. By definition of the Minkowski functional, $\psi_{S(\alpha, f)}(r^*) = 0$ implies that $r^*$ is in the recession cone of $S(\alpha, f)$. Thus, $\alpha r^* = 0$. Then $\tilde{x}$ belongs to $\inte(S(\alpha, f))$ iff $\lfloor \alpha f \rfloor< \alpha\tilde{x} < \lceil \alpha f \rceil$, which happens iff $\lfloor \alpha f \rfloor< \alpha(\tilde{x} + \lambda r^*) < \lceil \alpha f \rceil$ for all $\lambda \in \R$, which is equivalent to $(\tilde{x} + \lambda r^*) \in \inte(S(\alpha, f))$ for all $\lambda \in \R$. \end{proof}
%
			
			Define $R' = \{r \in R: \psi_S(r) > 0\}$ and $Q' = \{q \in Q : \pi_S(q) > 0\}$. Also define $\lambda_r = \bar{s}_r \psi_S(r)$ for every $r \in R$ and $\mu_q = \bar{y}_q \pi_S(q)$ for every $q \in Q$, and let $\nu = \sum_{r \in R'} \lambda_r + \sum_{q \in Q'} \mu_q$.
		 	
			
Notice that since $(\bar{x}, \bar{s}, \bar{y})$ satisfies $(\psi_S, \pi_S)$, we have $\nu =  \sum_{r \in R'} \lambda_r + \sum_{q \in Q'} \mu_q = \sum_{r \in R'} \bar{s}_r \psi_S(r) + \sum_{q \in Q'} \bar{y}_q \pi_S(q) = \sum_{r \in R} \bar{s}_r \psi_S(r) + \sum_{q \in Q} \bar{y}_q \pi_S(q) \ge 1$ (the third equality follows from the fact that $\psi_S(r) = 0$ for all $r \in R\setminus R'$ and $\pi_S(q) = 0$ for all $q \in Q \setminus Q'$, and the inequality follows from the fact that $(\bar{x}, \bar{s}, \bar{y})$ satisfies $(\psi_S, \pi_S)$). For every $r \in R'$, define the point $$U^r = \left(f + \frac{\nu r}{\psi_S(r)}, \left(\frac{\nu}{\psi_S(r)}\right) \cdot 1_r , 0\right)$$ where $1_r$ denotes the indicator function of $r$. Let $u^r = f + \frac{\nu r}{\psi_S(r)}$ be its first component ; similarly, for every $q \in Q'$ define $$V^q = \left(f + \frac{\nu q}{\pi_S(q)}, 0, \left(\frac{\nu}{\pi_S(q)}\right) \cdot 1_q\right)$$ and let $v^q = f + \frac{\nu q}{\pi_S(q)}$ be its first component. 
		 	
		 	Notice that all $U^r$ and $V^q$ belong to $\lpCorner$; moreover, using the definition of $\psi_S$ and $\pi_S$ and the fact that $\nu \ge 1$, we have that $u^r \notin \inte(S)$ and $v^q \notin \inte(S)$ for all $r \in R'$ and $q \in Q'$. Therefore, the points $U^r$ and $V^q$ belong to $\lpCorner  \cap \bar{D}(\alpha, f)$. 
		 	
		 	Now define \remove{in hindsight} $$\Delta = \left(\sum_{r \in R \setminus R'} \bar{s}_r r + \sum_{q \in Q \setminus Q'} \bar{y}_q q, \ \ \ \bar{s}|_{R \setminus R'}, \ \ \ \bar{y}|_{Q \setminus Q'}\right).$$ 
		Since $\psi_S = \pi_S$ are sub linear and nonnegative, we have that $0 \leq \psi_S(\sum_{r \in R \setminus R'} \bar{s}_r r + \sum_{q \in Q \setminus Q'} \bar{y}_q q) \leq \sum_{r \in R \setminus R'} \bar{s}_r \psi_S(r) + \sum_{q \in Q \setminus Q'} \bar{y}_q \pi_S(q) = 0$, where the equality follows from the fact that $\psi_S(r) = 0$ for all $r \in R\setminus R'$ and $\pi_S(q) = 0$ for all $q \in Q \setminus Q'$. Thus, the value of $\psi_S$ on the first component of $\Delta$ is zero. So Claim~\autoref{claim:belongsS} implies that $U^r + \Delta \in \lpCorner  \cap \bar{D}(\alpha, f)$ for all $r\in R'$ and $V^q + \Delta \in \lpCorner  \cap \bar{D}(\alpha, f)$ for all $q \in Q'$.
		
		Now notice that $\sum_{r \in R'} \frac{\lambda_r}{\nu}  + \sum_{q \in Q'} \frac{\mu_q}{\nu} = 1$ and we can write $(\bar{x}, \bar{s}, \bar{y})$ as follows:
		 	\begin{align}
		 		(\bar{x}, \bar{s}, \bar{y}) = \sum_{r \in R'} \frac{\lambda_r}{\nu} U^r + \sum_{q \in Q'} \frac{\mu_q}{\nu} V^q + \Delta = \sum_{r \in R'} \frac{\lambda_r}{\nu} (U^r+\Delta) + \sum_{q \in Q'} \frac{\mu_q}{\nu} (V^q + \Delta).
		 	\end{align}
Therefore, $(\bar{x}, \bar{s}, \bar{y})$ belongs to $\conv(\lpCorner \cap \bar{D}(\alpha, f))$.\end{proof}

		\section{A Geometric Lifting}
	
		The idea behind the lifting of cuts is to first obtain a (weaker) cut by ignoring some of the integrality of the variables and then strengthen it by incorporating back some of the discarded information. In our context, we can formally define lifting as follows: given a cut $(\psi, \pi)$ valid for $\contCorner(f, R, Q)$, we say that a cut $(\psi', \pi')$ valid for $\corner(f, R, Q)$ is a \emph{lifting} of $(\psi, \pi)$ if $\psi' = \psi$ and $\pi' \le \pi$. For example, the GMI function introduced in \autoref{s:basic_defs} can be seen as a lifting of a lattice-free cut, a fact that will be very important for our developments.
		
		\begin{fact}[\cite{ccz}] \label{fact:GMILifting}
			The GMI function $(\psi_{GMI}^f, \pi_{GMI}^f)$ is the lifting of the lattice-free cut $(\psi_{S(1,f)}, \pi_{S(1,f)})$. More precisely, $$\psi_{GMI}^f(r) = \psi_{S(1,f)}(r), \ \ \ \ \ \ \pi_{GMI}^f(q) = \min_{w \in \Z} \pi_{S(1,f)}(q + w).$$ 
		\end{fact}
		
We define a lifting procedure for lattice-free cuts. Given an integral function $\ell : Q \rightarrow \Z$ we consider the program
		\begin{equation} \tag{$\contCorner(f^0, R^0, Q^\ell)$} \label{eq:lifting}
			\begin{array}{rll}
			{x \choose x_{n+1}} &= {f \choose 0} + \sum_{r \in R} {r \choose 0} \cdot s{r \choose 0} + &\sum_{q \in Q} {q \choose \ell(q)} \cdot y{q \choose \ell(q)} \\
			x &\in \Z^{n}& \\
			x_{n+1} &\in \Z& \\
			s &\in \pFiniteSupp{R^0}, y \in \pFiniteSupp{Q^\ell}& 
			\end{array}
		\end{equation}	
		where we define $f^0 = (f,0)$, $R^0 = \{(r, 0) : r \in R\}$ and $Q^\ell = \{(q, \ell(q)) : q \in Q\}$.
		
		Intuitively, $\contCorner(f^0, R^0, Q^\ell)$ is a relaxation of $\corner(f, R, Q)$ and hence valid cuts for the former should be valid for the latter. Indeed, it is easy to check that given a lattice-free set $S \subseteq \R^{n+1}$ containing $(f,0)$ in its interior, the cut $(\psi^+_S, \pi^+_S)$ given by 
			\begin{align*}
				\psi^+_S(r) = \minkEmpty_{S - f}{r \choose 0} \ \ \ \forall r \in R, \ \ \ \pi^+_{S, \ell}(q) = \minkEmpty_{S - f}{q \choose \ell(q)} \ \ \ \forall q \in Q
			\end{align*}
			is valid for $\corner(f, R, Q)$. Moreover, keeping the set $S$ fixed, we can also look at the strongest possible cuts we can obtain by choosing different $\ell$'s; more precisely, the function $(\tilde{\psi}_S, \tilde{\pi}_S)$ given by
$$
			\begin{array}{c}
				\tilde{\psi}_S(r) = \psi^+_S(r) = \minkEmpty_{S - f}{r \choose 0}\ \ \ \forall r \in R,\\
				\\
				 \tilde{\pi}_{S}(q) = \inf\{ \pi^+_{S, \ell}(q) \mid \ell : Q \rightarrow \Z\} = \inf\left\{ \minkEmpty_{S - f}{q \choose \ell(q)} : \ell(q) \in \Z \right\} \ \ \ \forall q \in Q
			\end{array}
$$
			is valid for $\corner(f, R, Q)$.
			
The following technical lemma is well-known in the lifting literature. 

\begin{lemma}\label{lemma:min}
Let $S = S(\alpha, f) \subseteq \R^n$ be a lattice-free split set for some $\alpha \in \Z_f$. Define $\pi^*(r) = \inf\{\gamma_{S - f}(r+w) : w \in X\}$, where $X$ is a subset of $\Z^n$. Then, there exists $w^*\in X$ such that $\pi^*(r) = \gamma_{S-f}(r + w^*)$.
\end{lemma}

\begin{proof}
It can be verified that $$\gamma_{S-f}(r) = \max\left\{\frac{\alpha r}{\lceil \alpha f \rceil}, \frac{-\alpha r}{\alpha f - \lfloor \alpha f \rfloor}\right\} \triangleq \max\left\{c_1 \alpha r, - c_2 \alpha r\right\},$$ see for instance \cite{ccz}.

	For any $w \in X \subseteq \Z^n$, there exists $\lambda \in \R$ and $v \in \R^n$ such that $w = \lambda \alpha + v$ and $\alpha \cdot v = 0$. Thus, $\gamma_{S-f}(r + w) = \max\{c_1 \alpha r + c_1 \lambda \|\alpha\|^2, - c_2 \alpha r - c_2 \lambda \|\alpha \|^2\}$. Let $\Lambda = \{\lambda \in \R : \exists v \in \R^n\textrm{ such that } \lambda \alpha + v \in X, \alpha v = 0 \}$ be the projection of the $X\subseteq \Z^n$ onto the subspace spanned by $\alpha$; since $\alpha$ is rational, we have that $\Lambda$ is a subset of a lattice. With this definition, we have 
	\begin{align}
		\pi^*(r) = \inf\{\gamma_{S - f}(r+w) : w \in X\} = \inf\left\{ \max\{c_1 \alpha r + c_1 \lambda \lVert\alpha\rVert^2, - c_2 \alpha r - c_2 \lambda \lVert \alpha \rVert^2\} : \lambda \in \Lambda\right\}. \label{kcuts:eq:minkSplit}
	\end{align}	
	
		Now $\omega(\lambda) \triangleq \max\{c_1 \alpha r + c_1 \lambda \lVert\alpha\rVert^2, - c_2 \alpha r - c_2 \lambda \lVert \alpha \rVert^2\}$ is a piecewise linear convex function in $\lambda$ with two pieces, and $\Lambda$ is a subset of some lattice. Moreover, $\omega(\lambda) \geq 0$ for all $\lambda \in \R$. Thus, the infimum in the right-hand side of \eqref{kcuts:eq:minkSplit} is actually achieved. This proves the lemma.\end{proof}


	\section{Characterization of Split Cuts}

	\subsection{Split Cuts as Lifted Lattice-free-Split Cuts}\label{s:split-lattice-free}
	
	The main result in this section is the next lemma, which shows that every split cut for $\corner(f, R, Q)$ can be obtained via geometric lifting from a cut derived from a lattice-free split set. We say that $(\psi, \pi)$ {\em dominates} $(\psi', \pi')$ with respect to a set $X$ if $X \cap (\psi, \pi) \subseteq X \cap (\psi', \pi')$.

	\begin{lemma} \label{lemma:equivSplitLifted}
		A valid cut $(\psi, \pi)$ for $\corner(f, R, Q)$ is a split cut iff there is $\alpha \in \Z_f$ such that $(\psi, \pi)$ is dominated by the cut $(\tilde{\psi}_S, \tilde{\pi}_S)$ with respect to $\lpCorner(f, R, Q)$, where $S = S((\alpha,1), (f,0))$.
	\end{lemma}
	
	The heart of the argument for proving Lemma~\autoref{lemma:equivSplitLifted} is showing the equivalence between split cuts for $\corner(f, R, Q)$ and $\contCorner(f^0, R^0, Q^\ell)$; the idea is that we can simulate any split disjunction for the former by setting the value of $\ell$ appropriately and using a disjunction on the latter program which utilizes the new constraint $x_{n+1} = \sum_{q \in Q} \ell(q) \cdot y{q \choose \ell(q)}$. To make this precise, define the function $\Gamma^\circ_\ell : \R^R \times \R^Q \rightarrow \R^{R^0} \times \R^{Q^\ell}$ which maps $(\psi, \pi)$ into $(\hat\psi, \hat\pi)$ by setting $$\hat\psi(r,0) = \psi(r) \ \ \ \forall r \in R, \ \ \ \ \ \hat\pi(q, \ell(q)) = \pi(q) \ \ \ \forall q \in Q,$$ and the functional $\Gamma_{\ell} : \lpCorner(f, R, Q) \rightarrow \lpCorner(f^0, R^0, Q^{\ell})$ which maps a solution $(x, s, y)$ into $(x', s', y')$ by setting 
	\begin{align*}
		x'_i &= x_i, \ \ \ i = 1, 2, \ldots, n\\
		x'_{n+1} &= \sum_{q \in Q} \ell(q) y(q)\\
		s'(r, 0) &= s(r) \ \ \ \textrm{for all $r \in R$} \\
		y'(q, \ell(q)) &= y(q) \ \ \ \textrm{for all $q \in Q$}.
	\end{align*}
	
	The following can be readily verified.
	
	\begin{observation} \label{obs:isomorphism}
		Given any tuple $(x, s, y) \in \R^n \times \R^R \times \R^Q$, we have:
		\begin{enumerate}
			\item $(x, s, y)$ is a solution to $\lpCorner(f, R, Q)$ iff $\Gamma_\ell(x,s,y)$ is a solution to $\lpCorner(f^0, R^0, Q^{\ell})$.
			
			\item $(x,s,y)$ satisfies the cut $(\psi, \pi)$ if and only if $\Gamma_\ell(x,s,y)$ satisfies the cut $\Gamma^{\circ}_\ell(\psi, \pi)$.
		\end{enumerate}
	\end{observation}

		\begin{lemma}\label{lemma:integral_comb}
			The cut $(\psi, \pi)$ is a split cut for $\corner(f, R, Q)$ with respect to the disjunction $D(\alpha, \beta, f)$ iff $\Gamma^\circ_{\beta}(\psi, \pi)$ is a split cut for $\contCorner(f^0, R^0, Q^\beta)$ with respect to the disjunction $\bar{D}((\alpha, 1), (f,0))$. 
		\end{lemma}
		
		\begin{proof}
		Observe that $(x,s,y) \in D(\alpha, \beta, f)$ if and only if $\Gamma_{\beta}(x,s,y) \in \bar{D}((\alpha, 1), (f,0))$. From Observation~\autoref{obs:isomorphism}, $(x,s,y) \in \lpCorner(f, R, Q)$ if and only if $\Gamma_\beta(x,s,y) \in \lpCorner(f^0, R^0, Q^\beta)$. Thus, $(x,s,y) \in \lpCorner(f, R, Q) \cap D(\alpha, \beta, f)$ if and only if $\Gamma_\beta(x,s,y) \in \lpCorner(f^0, R^0, Q^\beta) \cap \bar{D}((\alpha, 1), (f,0))$. By Observation~\autoref{obs:isomorphism}, $(\psi, \pi)$ is valid for $\lpCorner(f, R, Q) \cap D(\alpha, \beta, f)$ if and only if $\Gamma^\circ_{\beta}(\psi, \pi)$ is valid for $\lpCorner(f^0, R^0, Q^\beta) \cap \bar{D}((\alpha, 1), (f,0))$. Thus, we are done.\end{proof}		
		
		We also need the following easy lemma, which states that $\Gamma^{\circ}_\ell$ preserves cut dominance.
		
		\begin{lemma} \label{lemma:cutIso}
			Let $(\psi, \pi)$ and $(\psi', \pi')$ be valid cuts for $\corner(f, R, Q)$.
		The cut $(\psi, \pi)$ dominates $(\psi', \pi')$ with respect to $\lpCorner(f, R, Q)$ iff the cut $\Gamma^\circ_\ell(\psi, \pi)$ dominates $\Gamma^{\circ}_\ell(\psi', \pi')$ with respect to $\lpCorner(f^0, R^0, Q^\ell)$. 
	\end{lemma}
	
	\begin{proof}
			The cut $(\psi, \pi)$ dominates $(\psi', \pi')$ with respect to $\lpCorner(f, R, Q)$ iff every point $(x, s, y) \in \lpCorner(f, R, Q)$ satisfying $(\psi, \pi)$ also satisfies $(\psi', \pi')$; from Observation~\autoref{obs:isomorphism} this happens iff every point $\Gamma_{\ell}(x, s, y)$ (for $(x, s, y) \in \lpCorner(f, R, Q)$) satisfying $\Gamma_{\ell}^\circ(\psi, \pi)$ also satisfies $\Gamma_{\ell}^\circ(\psi', \pi')$; but this happens iff the cut $\Gamma_{\ell}^\circ(\psi, \pi)$ dominates $\Gamma_{\ell}^\circ(\psi', \pi')$ with respect to $\lpCorner(f^0, R^0, Q^\ell)$.\end{proof}

		\begin{proofof}{Lemma \autoref{lemma:equivSplitLifted}}
			$(\Rightarrow)$ Suppose that $(\psi, \pi)$ is a split cut for $\corner(f, R, Q)$ with respect to the disjunction $D(\alpha, \beta, f)$. By Lemma~\autoref{lemma:integral_comb} we get that $(\psi', \pi') = \Gamma^{\circ}_\beta(\psi, \pi)$ is a split cut for $\contCorner(f^0, R^0, Q^\beta)$ with respect to the disjunction $\bar{D}((\alpha, 1), (f,0))$. From Theorem \autoref{thm:equivSplitLF}, we get that $(\psi', \pi')$ is dominated by $(\psi_S, \pi_S)$ with respect to $\lpCorner(f^0, R^0, Q^\ell)$. From Lemma \autoref{lemma:cutIso} we then have that $(\psi, \pi)$ is dominated by $(\Gamma_\beta^\circ)^{-1}(\psi_S, \pi_S)$ with respect to $\lpCorner(f, R, Q)$. Finally, notice that  $(\Gamma_\beta^\circ)^{-1}(\psi_S, \pi_S) = (\psi^+_S, \pi^+_{S, \beta})$; since $(\tilde{\psi}_{S}, \tilde{\pi}_{S})$ is at least as strong as $(\psi^+_S, \pi^+_{S, \beta})$, we have the first part of the lemma.
			
			$(\Leftarrow)$ By Lemma~\autoref{lemma:min}, $(\tilde\psi_S, \tilde\pi_S) = (\psi^+_{S}, \pi^+_{S, \ell})$ for some $\ell : Q \to \R$. Thus, it suffices to show that for all $\ell: Q \rightarrow \Z$ the cut $(\psi^+_{S}, \pi^+_{S, \ell})$ is a split cut for $\corner(f, R, Q)$. First notice that Theorem \autoref{thm:equivSplitLF} implies that $(\psi_S, \pi_S)$ is a split cut for $\contCorner(f^0, R^0, Q^\ell)$. Again since $(\psi^+_{S}, \pi^+_{S, \ell}) = (\Gamma^\circ_\ell)^{-1}(\psi_S, \pi_S)$, Lemma \autoref{lemma:integral_comb} gives that $(\psi^+_{S, \ell}, \pi^+_{S})$ is a split cut for $\corner(f, R, Q)$, thus concluding the proof of the lemma. 
		\end{proofof}


	
	\subsection{Linear Transformation of Lattice-free Cuts} \label{sec:tranfLatticeFree}

	An important tool to connect the GMI and $\alpha$-cuts with lattice-free split cuts for $\contCorner(f^0, R^0, Q^\ell)$ is to understand how the latter change when we apply a linear transformation to the split sets used to generate them. 
	
	The first observation follows directly from the definition of $\mink{S}{.}$ in \eqref{eq:mink}.
	
	\begin{lemma} \label{lemma:transfMink}
		Consider a convex set $S \subseteq \R^n$ with the origin in its interior. Then for every linear transformation $A : \R^m \rightarrow \R^n$ we have $\mink{S}{Ar} = \mink{A^{-1}(S)}{r}$ for all $r \in \R^m$.
	\end{lemma}

	Now we see how the description of a centered split $\bar{S}(u, f)$ changes when we apply a linear transformation to this set. \afterThesis{; as expected, the first parameter transforms ``contravariantly'' and the second transforms ``covariantly'' \mmnote{or vice-versa?}.}To make this formal, given a linear map $A : \R^m \rightarrow \R^n$, we use $A^t : \R^n \rightarrow \R^m$ to denote the \emph{adjoint} of $A$ (i.e., the unique linear map satisfying $A(u) \cdot v = u \cdot A^t(v)$).
	
	\begin{lemma} \label{lemma:transfSplit}
		Consider a centered split set $\bar{S}(u, f)$ with $u \in \Z^n$ and $f \in \R^n$, and a linear transformation $A : \R^m \rightarrow \R^n$. Then for every vector $f' \in A^{-1}(f)$, we have $$A^{-1}(\bar{S}(u, f)) = \bar{S}(A^t(u), f').$$ 
	\end{lemma}
	
\begin{proof}	Assume that $A^{-1}(f)$ is non-empty, otherwise there is nothing to prove, and consider $f' \in A^{-1}(f)$.	Recall that $\bar{S}(u,f) = S(u,f) - f = \{ x: \lfloor u f \rfloor \le u(x + f) \le \lceil uf \rceil\}$. Then we have
		\begin{align*}
			A^{-1}(\bar{S}(u,f)) &= \{ x \in \R^m : \lfloor u f \rfloor \le u \cdot (A(x) + f) \le \lceil uf \rceil\} \\
											&= \{ x \in \R^m : \lfloor u \cdot A(f') \rfloor \le u \cdot A(x + f') \le \lceil u \cdot A(f') \rceil\}\\
											&= \{ x \in \R^m : \lfloor A^t(u) \cdot f' \rfloor \le A^t(u) \cdot (x + f') \le \lceil A^t(u) \cdot f' \rceil\}\\
											&= \bar{S}(A^t(u), f').
		\end{align*} \end{proof}

	The next corollary states that we can use the geometric lifting to simulate the ``trivial lifting'' of a cut given by an interval, which is very handy given that we can construct the GMI function this way (recall Fact \autoref{fact:GMILifting}). 
	
	\begin{cor} \label{cor:trivialGeoLifting}
		For every centered split set $\bar{S}(u, f) \subseteq \R^1$, we have $$\minkEmpty_{\bar{S}(u, f)}(r + w) = \minkEmpty_{\bar{S}((u, u), (f,0))}{r \choose w} \ \ \ \ \forall r, w \in \R.$$
	\end{cor}
	
	\begin{proof}
		Let $A$ be the linear map $(x_1, x_2) \mapsto x_1 + x_2$ whose adjoint is $x_1 \mapsto (x_1, x_1)$, and let $f' = (f,0)$; apply Lemmas \autoref{lemma:transfMink} and \autoref{lemma:transfSplit}.
	\end{proof}

	\afterThesis{\mmnote{Add figure for the above lemma?}}

	In combination with the above result, the next corollary is useful to connect the definition of $\alpha$-cuts with our geometric lifting.

	\begin{cor} \label{cor:alphaCutInvariance}
		Consider $f \in \R^n$ and an integral vector $\alpha \in \Z_f$. Then
		\begin{align*}
			\minkEmpty_{\bar{S}((1,1), (\alpha f,0))}{\alpha q \choose w} = \minkEmpty_{\bar{S}((\alpha, 1), (f,0))}{q \choose w} \ \ \ \ \ \forall q \in \R^n, w \in \R.
		\end{align*}
	\end{cor}
	
	\begin{proof}
		Let $A$ be the linear map $(x_1, x_2, \ldots, x_{n+1}) \mapsto (\sum_{i = 1}^n \alpha_i x_i, x_{n+1})$, whose adjoint is $(x_1, x_2) \mapsto (\alpha_1 x_1, \alpha_2 x_1, \ldots, \alpha_n x_1, x_2)$, and let $f' = (f,0)$; apply Lemmas \autoref{lemma:transfMink} and \autoref{lemma:transfSplit}.
	\end{proof}


	\subsection{$\alpha$-cuts as Lifted Lattice-free Split Cuts}\label{alpha-lattice-free}

		Now we use the tools from the previous section to show that $\alpha$-cuts for $\corner(f, R, Q)$ are liftings of lattice-free split cuts for $\contCorner(f, R, Q)$. Notice the similarity of the statement below and Lemma \autoref{lemma:equivSplitLifted}, which shows that every split cut for $\corner(f, R, Q)$ is a lifting of a lattice-free split cut for $\contCorner(f, R, Q)$; in fact, it is through these liftings of lattice-free split cuts that we can show the equivalence of $\alpha$-cuts and split cuts for $\corner(f, R, Q)$ in the sequel.

		\begin{lemma} \label{lemma:kCutLifting}
			For every $\alpha \in \Z_f$, the $\alpha$-cut can be obtained using the geometric lifting procedure. More precisely, defining $S = S((\alpha, 1), (f,0))$ we have $$(\psi^f_{\alpha}, \pi_{\alpha}^f) = (\tilde{\psi}_{S}, \tilde{\pi}_{S}).$$			
		\end{lemma}
		
		\begin{proof}
			Recall that from the definition of $\alpha$-cuts and Fact \autoref{fact:GMILifting} that $\psi_\alpha^f(r) = \psi_{\GMI}^{\alpha f}(\alpha r) = \mink{\bar{S}(1, \alpha f)}{\alpha r}$. Then using Corollary \autoref{cor:trivialGeoLifting} with $u = 1$ and $w=0$ and Corollary \autoref{cor:alphaCutInvariance} with $w = 0$, we get that $$\psi_\alpha^f(r) = \mink{\bar{S}(1, \alpha f)}{\alpha r} \stackrel{\textrm{Cor \autoref{cor:trivialGeoLifting}}}{=} \minkEmpty_{\bar{S}((1,1), (\alpha f,0))}{\alpha r \choose 0} \stackrel{\textrm{Cor \autoref{cor:alphaCutInvariance}}}{=} \minkEmpty_{\bar{S}((\alpha,1), (f,0))}{r \choose 0} = \tilde{\psi}_{S((\alpha, 1), (f,0))}(r).$$ 
			
			Similarly, recall that $$\pi_\alpha^f(q) = \pi_{\GMI}^{\alpha f}(\alpha q) = \min_{w \in \Z} \psi_{\GMI}^{\alpha f}(\alpha q + w) = \min_{w \in \Z} \mink{\bar{S}(1, \alpha f)}{\alpha q + w}.$$ Then employing Corollaries \autoref{cor:trivialGeoLifting} and \autoref{cor:alphaCutInvariance} we get that $$\pi_\alpha^f(q) =  \min_{w \in \Z} \mink{\bar{S}(1, \alpha f)}{\alpha q + w} \stackrel{\textrm{Cor \autoref{cor:trivialGeoLifting}}}{=}  \min_{w \in \Z} \minkEmpty_{\bar{S}((1,1), (\alpha f,0))}{\alpha q \choose w} \stackrel{\textrm{Cor \autoref{cor:alphaCutInvariance}}}{=} \min_{w \in \Z} \minkEmpty_{\bar{S}((\alpha,1), (f,0))}{q \choose w} = \tilde{\pi}_{S((\alpha, 1), (f,0))}(q),$$ 
		concluding the proof of the lemma.
		\end{proof}


	\subsection{Proof of Theorem \autoref{thm:equivalence}}
%
		
			First we prove that every $\alpha$-cut is a split cut for $\corner(f, R, Q)$. Consider some $\bar{\alpha} \in \Z_f$ and the cut $(\psi_{\bar{\alpha}}^f, \pi_{\bar{\alpha}}^f)$. From Lemma \autoref{lemma:kCutLifting} we have that $(\psi_{\bar{\alpha}}^f, \pi_{\bar{\alpha}}^f)$ equals $(\tilde{\psi}_S, \tilde{\pi}_S)$ for $S = S((\bar{\alpha}, 1), (f, 0))$. The validity of the latter for $\corner(f, R, Q)$ then implies the validity of the former,  and using the ``if'' part of Lemma \autoref{lemma:equivSplitLifted} concludes the proof of this part.
			
			For the second part, consider a split cut $(\psi, \pi)$ for $\corner(f, R, Q)$. From Lemma \autoref{lemma:equivSplitLifted} there exists $\alpha$ in $\Z_f$ such that $(\psi, \pi)$ is dominated by $(\tilde{\psi}_{S((\alpha, 1), (f,0))}, \tilde{\pi}_{S((\alpha, 1), (f, 0))})$ with respect to $\lpCorner(f, R, Q)$. By Lemma~\autoref{lemma:kCutLifting}, this implies that $(\psi, \pi)$ is dominated by $(\psi^f_\alpha, \pi^f_\alpha)$ with respect to $\lpCorner(f, R, Q)$.
		

	
	\section{Applications}
	
		\subsection{Universality of Infinite Relaxation with Respect to Split Closure}
	
		We now use Theorem \autoref{thm:equivalence} to prove Theorem \autoref{thm:universalIR}. Before we do so we need the following technical lemma that gives a more natural characterization of the domination relationship between valid cuts.
		
		\afterThesis{
		\mmnote{The proof of the next lemma is pretty pedestrian and uses the assumption that one of the cuts is non-negative; probably the general statement can be proved using conic duality, but might be too much trouble}}
		
		\begin{lemma} \label{lemma:dominanceCorner}
			Consider a non-empty corner relaxation $\corner = \corner(f, R, Q)$. Let $(\psi, \pi)$ and $(\psi', \pi')$ be valid cuts for $\corner$. Then $$\corner_{LP}(f, R, Q)\cap (\psi, \pi) \subseteq \corner_{LP}(f, R, Q)\cap (\psi', \pi')\quad \Longleftrightarrow \quad(\psi, \pi) \le (\psi', \pi').$$
In other words, $(\psi, \pi)$ dominates $(\psi', \pi')$ with respect to $\corner_{LP}(f, R, Q)$ iff $(\psi, \pi) \le (\psi', \pi')$. 
		\end{lemma}\afterThesis{\mmnote{Be careful: proof uses non-negativity of $(\psi, \pi)$ (which is instantiated with $k$-cuts, so we are fine anyway)}}
		
	\begin{proof}	The ``if'' part is trivial, so we prove only the ``only if'' part. We prove the contrapositive:  assume that $(\psi, \pi) \not\le (\psi', \pi')$ and show that $\corner_{LP}(f, R, Q)\cap (\psi, \pi) \not\subseteq \corner_{LP}(f, R, Q)\cap (\psi', \pi')$. We consider the case where $\psi \not\le \psi'$, the other one is analogous. So let $\bar{r} \in R$ be such that $\psi(\bar{r}) > \psi'(\bar{r})$. 
			
			First consider the case where $\psi'(\bar{r}) \ge 0$, and hence $\psi(\bar{r}) > 0$. Then we construct the solution $$(\bar{x},\bar{s},\bar{y}) = \left(f + \frac{1}{\psi(\bar{r})} \bar{r}, \ \ \frac{\mathbbm{1}_{\bar{r}}}{\psi(\bar{r})}, \ \ 0\right),$$ where $\mathbbm{1}_{\bar{r}} : R \rightarrow \R$ is the function taking value $\mathbbm{1}_{\bar{r}}(\bar{r}) = 1$ and $\mathbbm{1}_{\bar{r}}(r) = 0$ for all other $r$. Notice that $(\bar{x},\bar{s},\bar{y})$ belongs to $\corner_{LP}(f, R, Q) \cap (\psi, \pi)$ but not to $\corner_{LP}(f, R, Q) \cap (\psi', \pi')$. This concludes the proof in this case.
		
			Now consider the case where $\psi'(\bar{r}) < 0$. Let $(\bar{x}, \bar{s}, \bar{y})$ be a feasible solution for $\corner(f, R, Q)$; in particular we have $(\bar{x}, \bar{s}, \bar{y}) \in \corner_{LP}(f, R, Q) \cap (\psi, \pi)$. Now let $\lambda > 0$ and consider the solution 
			\begin{align*}
				(x', s', y') = (\bar{x} + \lambda \bar{r}, \bar{s} + \lambda \mathbbm{1}_{\bar{r}}, \bar{y}).
			\end{align*}
			Since valid functions are non-negative by assumption (and in particular $\psi(\bar{r}) \ge 0$), notice that we still have $(x', s', y') \in \corner_{LP}(f, R, Q) \cap (\psi, \pi)$ for all $\lambda \geq 0$. However, since $\psi'(\bar{r}) < 0$, setting $\lambda$ large enough gives that $(x', s', y')$ does not belong to $\corner_{LP}(f, R, Q) \cap (\psi', \pi')$. This concludes the proof of the lemma. \end{proof}


		\begin{proofof}{Theorem \autoref{thm:universalIR}}
			Let $(\psi, \pi)$ be a split cut for $\corner(f, R, Q) \neq \emptyset$. The second part of Theorem \autoref{thm:equivalence} guarantees that there is $\bar{\alpha} \in \Z_f$ such that $(\psi^f_{\bar{\alpha}}|_{R}, \pi^f_{\bar{\alpha}}|_{Q})$ dominates $(\psi, \pi)$ with respect to $\corner_{LP}(f, R, Q)$. Lemma \autoref{lemma:dominanceCorner} above then gives that $(\psi^f_{\bar{\alpha}}|_{R}, \pi^f_{\bar{\alpha}}|_{Q}) \le (\psi, \pi)$.
			
			Then define the cut $(\psi', \pi')$ for $\corner(f, \R^n, \R^n)$ such that $(\psi'|_{R}, \pi'|_{Q}) = (\psi, \pi)$ and $(\psi'|_{\R \setminus R}, \pi'|_{\R \setminus Q}) = (\psi^f_{\bar{\alpha}}|_{\R \setminus R}, \pi^f_{\bar{\alpha}}|_{\R \setminus Q})$. We then have that: (i) $(\psi', \pi')$ is a split cut for $\corner(f, R, Q)$, because $(\psi', \pi') \ge (\psi^f_{\bar{\alpha}}, \pi^f_{\bar{\alpha}})$ and $(\psi^f_{\bar{\alpha}}, \pi^f_{\bar{\alpha}})$ is a split cut for $\corner(f, R, Q)$ (by the first part of Theorem \autoref{thm:equivalence}); (ii) $(\psi, \pi)$ is a restriction of $(\psi', \pi')$. This concludes the proof of the theorem. 
		\end{proofof}

	
	\subsection{Pure Integer Program with Weak Split Closure}
	
	In this section we prove Theorem \autoref{thm:badIP} by exhibiting a pure integer corner relaxation with a weak split closure; throughout, we will only work with corner relaxations $\corner(f, R, Q)$ where both $R$ and $Q$ are finite sets of rational vectors.
		
	Recall that the split closure of a corner relaxation $\corner(f, R, Q)$, denoted by $\sClosure(f, R, Q)$, is the set of all points in $\lpCorner(f, R, Q)$ that satisfy all split cuts for $\corner(f, R, Q)$. Using Theorem \autoref{thm:equivalence} we can describe this closure as 
	\begin{align}
		\sClosure(f, R, Q) = \lpCorner(f, R, Q) \cap \left(\bigcap_{\alpha \in \Z_f} (\psi^f_\alpha, \pi^f_\alpha) \right). \label{eq:splitClosureKCuts}
	\end{align}
	
	
	\paragraph{Bad example} Now we present the family of integer corner polyhedron with a weak split closure. Consider the following family parameterized by a rational number $\epsilon > 0$:
	\begin{align*} \label{badIP} \tag{$\corner_\epsilon$}
		x = \frac{1}{2} + \left(\frac{1}{2} + \frac{\epsilon}{2}\right) y_1 + \left(\frac{1}{2} + \epsilon\right) y_2\\
		y_1, y_2 \ge 0 \\
		x, y_1, y_2 \in \mathbb{Z}.
	\end{align*}
	To make things more clear, we define $\bar{f} = 1/2$, $\bar{q}^1 = (\frac{1}{2} + \frac{\epsilon}{2})$ and $\bar{q}^2 = (\frac{1}{2} + \epsilon)$, so formally $\corner_\epsilon = \corner(\bar{f}, \emptyset, \{\bar{q}^1, \bar{q}^2\})$.
	
	Let $\sClosure_\epsilon$ denote the split closure of $\corner_\epsilon$; we claim that there is a gap of $1/12 \epsilon$ between these two sets when minimizing $y_1 + y_2$, i.e., they satisfy Theorem \autoref{thm:badIP}. The intuition behind this construction is the following. The equivalence in Theorem \autoref{thm:equivalence}, via equation \eqref{eq:splitClosureKCuts} allow us to focus solely on understanding the effect of $\alpha$-cuts on the above program. Since the latter only has integrally constrained non-basic variables, this means focusing on the functions $\pi^{\bar{f}}_{\alpha}$. Observing the behavior of the function $\pi_{\alpha}^{\bar{f}}$, we see that it essentially has ``high'' value (close to 1) for inputs close to $1/2$; while this is not exactly true for large $\alpha \in \Z_{\bar{f}}$, our particular choice of $\bar{q}^1$ and $\bar{q}^2$ guarantees that \emph{one of} $\pi^{\bar{f}}_\alpha(\bar{q}^1)$ or $\pi^{\bar{f}}_\alpha(\bar{q}^2)$ is reasonably large for every choice of $\alpha \in \Z_{\bar{f}}$.
	
	\begin{lemma}
		For every $\epsilon > 0$, we have that $\max\left\{\pi^{\bar{f}}_{\alpha}(\bar{q}^1), \pi^{\bar{f}}_{\alpha}(\bar{q}^2)\right\} \ge 1/3$ for all $\alpha \in \Z_{\bar{f}}$. 
	\end{lemma}
	
	\begin{proof}
		By our choice of $\bar{f} = 1/2$, notice that for every $\alpha \in \Z_{\bar{f}}$ we have that the fractional part $[\alpha \bar{f}]$ equals $1/2$. 
Thus, using the definition of $\pi^{\bar{f}}_\alpha$ given by equations \eqref{eq:defGMIPi} and \eqref{eq:defAlphaCutPi}, we get that $\pi^{\bar{f}}_\alpha$ takes the form
		\begin{align}
			\pi^{\bar{f}}_{\alpha}(q) = \left\{ \begin{array}{lr} 2 [\alpha q] & \textrm{, if } [\alpha q] \le \frac{1}{2} \\ 2 - 2 [\alpha q] & \textrm{, if } [\alpha q] > \frac{1}{2}  \end{array}\right. \label{eq:explicitKCut}
		\end{align}
	The next claim gives some control on the behavior of the fractional part $[\alpha \bar{q}^i]$.	
	\begin{myclaim} \label{claim:alphaQ}
		For every $\alpha \in \Z_{\bar{f}}$ either $[\alpha \bar{q}^1]$ or $[\alpha \bar{q}^2]$ lies in the interval $[1/6,5/6]$. 		
	\end{myclaim}
	
\begin{proof} Take $\alpha \in \Z_{\bar{f}}$. Since $\bar f = \frac12$, $\alpha$ is an odd integer. Therefore, we have that $[\alpha \bar{q}^1] = [1/2 + \frac{\alpha\epsilon}{2}] = [1/2 + [\frac{\alpha\epsilon}{2}]]$. Using the latter, it is easy to verify that $[\alpha \bar{q}^1] \in [1/6, 5/6]$ if and only if $[k\epsilon /2] \notin (1/3, 2/3)$. Similarly, we have that $[\alpha \bar{q}^2] \in [1/6, 5/6]$ if and only if $[k\epsilon] \notin (1/3, 2/3)$.
		
		To conclude the proof, it suffices to show that $[\alpha \epsilon/2] \in (1/3, 2/3)$ implies $[\alpha \epsilon] \notin (1/3,2/3)$. For that notice that when $[\alpha \epsilon/2] \in (1/3,1/2]$, this implies $[\alpha \epsilon] \in (2/3, 1]$ and hence $[\alpha \epsilon] \notin (1/3, 2/3)$. On the other hand, when $[\alpha \epsilon/2] \in [1/2,2/3)$ we have $[\alpha \epsilon] \in [0, 1/3)$, again reaching the conclusion $[\alpha \epsilon] \notin (1/3, 2/3)$. \end{proof}

	Take $\alpha \in \Z_{\bar{f}}$ and use the above claim to get $i \in \{1, 2\}$ such that $[\alpha \bar{q}^i] \in [1/6, 5/6]$. If $[\alpha \bar{q}^i] \in [1/6, 1/2]$, then employing equation \eqref{eq:explicitKCut} we get that $\pi^{\bar{f}}_\alpha(\bar{q}^i) = 2 [\alpha \bar{q}^i] \ge 1/3$; if $[\alpha \bar{q}^i] \in [1/2, 5/6]$, the same equation gives $\pi^{\bar{f}}_\alpha(\bar{q}^i) = 2 - 2 [\alpha \bar{q}^i] \ge 2 - 10/6 = 1/3$. This concludes the proof of the lemma.
	\end{proof}	

	Using the above lemma, we see that the point $(\bar{x}, \bar{y}_1, \bar{y}_2)$ given by $\bar{y}_1 = \bar{y}_2 = 3$ and $\bar{x} = \bar{f} + \bar{q}^1 \bar{y}_1 + \bar{q}^2 \bar{y}_2$ belongs to the linear relaxation of \eqref{badIP} and satisfies all $\alpha$-cuts for it; therefore, this point belongs to the split closure of $\sClosure_\epsilon$. This directly gives the following. 
	
	 \begin{lemma} \label{lemma:badSplit}
	 		The optimal value of minimizing $y_1 + y_2$ over the split closure $\sClosure_\epsilon$ is at most $6$.
	 \end{lemma}
	
	Now in order to show the weakness of the split cuts, we show that the optimal value of minimizing $y_1 + y_2$ over the whole of $\corner_\epsilon$ is much larger. 
	
	\begin{lemma} \label{lemma:strongIP}
		The optimal value of minimizing $y_1 + y_2$ over $\corner_\epsilon$ is at least $1/2\epsilon$.
	\end{lemma}

	\begin{proof}
		We can rewrite the equation in \eqref{badIP} as $$x = \frac{1}{2} (1 + y_1 + y_2) + \epsilon \left( \frac{y_1}{2} + y_2 \right).$$ Since for every solution in $\corner_\epsilon$ we have $y_1, y_2 \in \Z_+$, this implies that $\frac{1}{2}(1 + y_1 + y_2) \in \Z_+/2$; since in such solution we also have $x \in \Z$, this implies that $\epsilon (y_1/2 + y_2) \in \Z_+/2$. Moreover, we need to have one of $y_1$ or $y_2$ strictly positive (and the other non-negative), so the previous observation actually implies that $\epsilon (y_1/2 + y_2) \ge 1/2$ is satisfied by all feasible solutions. Such solutions then satisfy the weaker inequality $y_1 + y_2 \ge 1/2 \epsilon$; this concludes the proof of the lemma.
	\end{proof}
		
	Theorem~\autoref{thm:badIP} then follows directly from Lemmas \autoref{lemma:badSplit} and \autoref{lemma:strongIP}.

	\bigskip
	\noindent \textbf{Acknowledgments:} We would like to thank Gerard Cornu\'{e}jols, Sanjeeb Dash, Santanu Dey, Fran\c{c}ois Margot and R. Ravi for comments on the paper. 


	\bibliographystyle{plain}
	\bibliography{ip}		
		
\end{document}